\documentclass[12pt]{article}
\usepackage[margin=1.35in]{geometry}
\usepackage{amsmath,yhmath}
\usepackage{amsthm}
\usepackage{amssymb}
\usepackage{amsfonts}
\usepackage{bbm} 
\usepackage{cite}
\usepackage{tikz}

\usepackage{enumitem}

\usepackage{subfig}

\usepackage{caption}

\usepackage{microtype}
\usepackage{lmodern}

\usepackage{tikz}
\usetikzlibrary{shapes,arrows,graphs,graphs.standard,shapes.misc}
\usetikzlibrary{matrix,decorations.pathreplacing, calc, positioning,fit}
\usetikzlibrary{shapes.geometric}
\usetikzlibrary{cd}

\newcommand{\xc}[1]{\vspace{.3cm}

\noindent {\em #1} }

\newtheorem{definition}{Definition}

\newtheorem{theorem}{Theorem}
\newtheorem{lemma}{Lemma}
\newtheorem{corollary}[theorem]{Corollary}
\newtheorem{proposition}[theorem]{Proposition}

\makeatletter
\xdef\@endgadget#1{{\unskip\nobreak\hfil\penalty50\hskip1em\hbox{}\nobreak\hfil#1\parfillskip=0pt\finalhyphendemerits=0\par}}
\newcommand\@Endofsymbol{$\triangledown$}
\newcommand\Endofremark{\@endgadget{\@Endofsymbol}}
\makeatother

\newcommand{\R}{\mathbb{R}}

\newcommand{\rank}{\operatorname{rank}}

\newcommand{\cL}{\mathcal{L}}

\renewcommand{\i}{$(i)$ }
\newcommand{\ii}{$(ii)$ }

\newcounter{cAss}
\newcounter{cAssSaved}
\newcommand\Ass[1]{\ensuremath{\boldsymbol{\mathcal A}_{\text{\hspace{0.75pt}\bf#1}}}}
\newlength\asswidth
 {\begin{list}{\Ass{\arabic{cAss}}:}{%
   \setcounter{cAssSaved}{\thecAss}\usecounter{cAss}\setcounter{cAss}{\thecAssSaved}
   \setlength\topsep{.667ex plus 2pt minus 2pt}
   \setlength\itemsep\topsep
   \settowidth\asswidth{\Ass{\arabic{cAss}}:}
   \addtolength\asswidth\labelsep
   \setlength\leftmargin\asswidth
   \setlength\labelwidth\asswidth}}%
 {\end{list}}

\usepackage[color=yellow, textsize=small, textwidth=\marginparwidth, draft]{todonotes}   
 
\usepackage{xcolor}
\definecolor{forestgreen}{rgb}{0.13, 0.55, 0.13}
\definecolor{orange}{rgb}{1,0.49,0}

\newcounter{para}[section]

\graphicspath{{Figures/}}

\begin{document}

\title{Visible and hidden observables in super-linearization}
\author{Mohamed-Ali Belabbas
}
\date{}\maketitle
\begin{abstract}
We call a system super-linearizable if it admits  finite-dimensional embedding as a linear system---known as a finite-dimensional Koopman embedding; said otherwise, if its dynamics can be linearized by adding a finite set of observables. We introduce the notions of visible and hidden observables for such embeddings which, roughly speaking, are the observables that explicitly appear in the original system and the ones that do not, but yet are necessary for its embedding. Distinct embeddings can have different numbers of hidden and visible observables. In this paper, we derive a tight lower bound for the number of visible observables of a system among all its super-linearizations. 
\end{abstract}
\fussy

\section{Introduction}

We consider  control systems of the form
\begin{equation}\label{eq:mainsys}
\dot x = f(x)+ug(x),	
\end{equation}
where $x \in \R^n$, $f,g$ are smooth vector fields in $\R^n$, and their embeddings in higher dimensional state spaces~\cite{kowalski1991nonlinear}.  The goal  of the embedding is to utilize the additional degrees of freedom to linearize the system without recourse to changes of variables. 
Namely, we say that system~\eqref{eq:mainsys} admits a linear finite-dimensional embedding, or a {\em super-linearization} or a finite-dimensional Koopman linearization, if there exists $m\geq 0$ functions, called  observables, which when adjoined to the original system permit its linearization. The idea of embedding non-linear dynamics, in potentially  infinite-dimensional state-spaces, as a linear system dates back at least to the works of Koopman~\cite{koopman1931hamiltonian} and Carleman~\cite{carleman1932application}. These ideas have be used in geometric control~\cite{brockett1976volterra,brockett2014early} and more recently as a cornerstone of data-driven control~\cite{mauroy2020koopman}.

A typical example~\cite{brunton2016koopman} of system with a finite-dimensional embedding is the following two-dimensional system
\begin{equation}\label{eq:ex1}
\begin{cases}
\dot x &= -x+y^2 +u \\
\dot y &= -y	
\end{cases}
\end{equation}
The vector field contains a linear term $-x e_1 - ye_2$, where $\{e_i\}_{i=1,\ldots,n}$ is the canonical basis of $\R^n$, and a nonlinear term $y^2 e_1$. After adjoining the observable $p:=y^2$, whose total time derivative is given by $\dot p = 2y\dot y = -2y^2 =-2p$, we obtain the three-dimensional {\em linear} system
\begin{equation}\label{eq:ex2}
\begin{cases}
\dot x &= -x+p +u \\
\dot y &= -y\\
\dot p &=-2p.	
\end{cases}
\end{equation}
Define the projection map $\Pi(x,y,p):=(x,y)$. We call $\Pi$ a {\em standard projection} and~\eqref{eq:ex2} a linear {\em super-linearization} with {\em observable} $p:(x,y) \mapsto y^2$. We precisely define these notions below.  We see that solutions of~\eqref{eq:ex2} with initial conditions $(x_0,y_0,y_0^2)$ are mapped by $\Pi$ to solutions of~\eqref{eq:ex1} with initial conditions $(x_0,y_0)$.

  We will say that an observable is visible if it appears explicitly in the nonlinear dynamics that we seek to linearize. A precise definition is given below. In the example~\eqref{eq:ex1}, the only observable is $y^2$, and it is thus visible. In order to obtain a linear embedding of the system, we sometimes need to introduce additional observables which do not appear explicitly in the original dynamics. We call these observables hidden. We illustate the notion in the simple case of the system
\begin{equation}\label{eq:nonunimdyn}
\begin{cases}
\dot x &= -x+y+y^2+y^3\\
\dot y &= y.	
\end{cases}
\end{equation}
On the one hand, introducing the observables $p_1 = y^2$ and $p_2=y^3$, we get that
$$
\begin{cases}
\dot p_1 &= 2p_1\\
\dot p_2 &= 3p_2	
\end{cases}
$$
and thus the system can be super-linearized to
$$
\begin{cases}
\dot x &= -x+y+p_1+p_2\\
\dot y &= y\\
\dot p_1 &= 2p_1\\
\dot p_2 &= 3p_2.
\end{cases}
$$
There are two observables for this super-linearization, $p_1$ and $p_2$, and we refer to both as {\em visible}, as they appear explicitly in the right-hand-side of the dynamics of the original system.

In contrast, set $p_1:= y^2 + y^3$, $p_2 = y^2$ and $p_3=y^3$, we get that
$$
\begin{cases}
\dot p_1 &= 2p_2+3p_3\\
\dot p_2 &= 2p_2\\	
\dot p_3 &= 3p_3
\end{cases}
$$
 and we conclude that the system can be super-linearized as
 
$$
\begin{cases}
\dot x &= -x+y+p_1\\
\dot y &= y\\
\dot p_1 &= 2p_2+3p_3\\
\dot p_2 &= 2p_2\\	
\dot p_3 &= 3p_3.
\end{cases}
$$
In this case, $p_1$ is a visible observable, and $p_2$ and $p_3$ are hidden observables: they are necessary for the super-linearization of the system, but do not appear explicitly in the original dynamics.

The notions of visible and hidden observables are fairly natural ones, and it is important to realize that their number depends on a particular choice of super-linearization. We provide an example of this fact below.  This raises the problem of establishing the least number of visible or hidden observables amongst {\em all super-linearizations} of a given system. We provide an answer to this question here for the case of visible observables.

The proof relies on the following steps. The first is to establish a standard form for super-linearizable systems, a form which relates the original dynamics to a super-linearization of it. This is done in Theorem~\ref{th:main1}. This form will allow us to introduce what we call below the $G$-matrix of the super-linearization, and subsequently define the notions of visible and hidden observables precisely.  In the second step, we establish some elementary operations that can be performed on a super-linearization to obtain a different one. This is done in Lemma~\ref{lem:changevarQ} and Propositions~\ref{prop:mh},\ref{prop:linearcorrection} and~\ref{prop:reducingLIm}. We then define the notion of reduced visible form of a super-linearization and show that for any given system, all of its super-linearization in reduced-visible forms have $G$-matrices of the same rank. This is done in Proposition~\ref{prop:invarkG}. In the last step,  relating the rank of the $G$-matrix to the number of visible observables, we establish the lower bound sought. This is done in Theorem~\ref{th:mv}.

\paragraph*{Notation and conventions} Throughout the paper
$$
(f,g) \mbox{ denotes the system }\dot x = f + u g.
$$ 
A smooth embedding $\psi:\R^n \to \R^{n'}$ with $n' > n$ is a smooth map for which there exists a smooth inverse $\psi': \psi(\R^{n}) \to \R^n$ with the property that $\psi \circ \psi ' = I$ and $\psi' \circ \psi = I$.

 We let $e^{t(f+ug)}x_0$ be the solution at time $t$ of~\eqref{eq:mainsys} with initial state $x_0$ and control $u$.  

Given positive integers $n$ and $m$, we denote by $\Pi:\R^{n+m} \to \R^n$ the standard projection onto the first $n$ variables and $\bar \Pi:\R^{n+m} \to \R^m$ the standard projection onto the last $m$ variables, i.e., $$\Pi(z_1,\ldots,z_{n+m})=(z_1,\ldots,z_n)\mbox{ and }\bar \Pi(z_1,\ldots,z_{n+m})=(z_{n+1},\ldots,z_{n+m}).$$ The integers $n$ and $m$ will be clear from the context. 
With some abuse of notation, for $z \in \R^{n + m}$ we also set  $z_1:= \Pi(z) \in \R^n$ and $z_2:= \bar \Pi(z) \in \R^m$, so that $z= \begin{bmatrix} z_1^\top & z_2^\top \end{bmatrix}^\top .$ In order to simplify the notation, we also write $z=\begin{pmatrix} z_1 & z_2 \end{pmatrix}$, with the understanding that $z, z_1$ and $z_2$ are column vectors.
Given a map $p:\R^n \to \R^m$, we say that $p$ has a constant term if $p(0) \neq 0$ and that $p$ has a linear term if $\frac{d}{dx}|_{x=0}p(x) \neq 0$. 
 
We let
$$\iota(x):=\begin{pmatrix} x & p(x)\end{pmatrix},$$
where the function $p$ will be clear from context.  
Given a smooth map $\Pi$, we denote by $d\Pi$ its Jacobian.

We let $I$ be the identity matrix whose dimension is determined by the context or explicitly indicated via an index. 
An {\em affine control system} is a controlled differential equation of the form
$$\dot z = A_\ell z +B_\ell u + D_\ell,
$$
for a matrix $A_\ell$ and vectors $B_\ell,D_\ell$ of the appropriate dimensions for the equation to be well-defined.
We refer to the affine control system above as {\em the triple} $(A_\ell, B_\ell, D_\ell)$.

We denote by $C(A,B)$ and $O(A,G)$ the controllability and observability matrices, respectively, associated with the system
\begin{equation}\label{eq:assoclinsys}
\begin{cases}\dot x &= Ax+Bu\\  y&=Gx.
\end{cases}
\end{equation}

\section{Statement of the main results} 

To state the main results, we first precisely define super-linearizations of systems. 

\begin{definition}[Super-linearization]\label{def:pifeedbackequiv}We say that the system $(f, g)$ in $\R^n$ is {\em smoothly embedded as a finite-dimensional affine system}---or {\em super-linearized} to---$(A_\ell,B_\ell, D_\ell)$  with $A_\ell \in \R^{(n+m)\times(n+m)}$ and $B_\ell, D_\ell \in \R^{n+m}$ if there exists a smooth map $p:\R^n \to \R^{m}$   so that for all $x_0 \in \R^n$ and control $u(t)$, it holds that
\begin{equation}\label{eq:equivalence1}
\Pi\left(e^{t(A_\ell z+B_\ell u +D_\ell)}z_0 \right) = e^{t(f+ug)}x_0	\mbox{ with } z_0=\begin{pmatrix}x_0 & p(x_0)\end{pmatrix}
\end{equation} 
as long as the solutions exist.
We call the functions $p_j:\R^n \to \R$, $j=1,\ldots,m$,  the {\em observables} and the data of $\cL:=(A_\ell, B_\ell, D_\ell,p)$ an affine finite-dimensional  embedding or super-linearization. If $D_\ell=0$, we call them a linear finite-dimensional embedding.
\end{definition}

We can express the relation~\eqref{eq:equivalence1} as the following commutative diagram
\begin{center}
\begin{tikzcd}[column sep=huge, row sep=huge]
  \R^n \arrow[r, "e^{t(f+ug)}"] \arrow[d,  "(I \,\ p)" left ] & \R^{n} \\
  \R^{n+m} \arrow[r,  "e^{t(A_\ell z+B_\ell u+D_\ell)}" below] & \R^{n+m} \arrow[u,  "\Pi" right]
  \end{tikzcd}
\end{center}
Returning to the system~\eqref{eq:ex1}, we see it can be super-linearized with $m=1$ and $p(x,y)=y^2$ to the affine system
$$
A_\ell=\begin{bmatrix}
-1 & 0 & 1\\
0 & -1 & 0\\
0 & 0 &-2	
\end{bmatrix}, B_\ell =\begin{bmatrix}
1\\0\\0	
\end{bmatrix}, D_\ell = 0.
$$

We now state the main results of the paper. We start with the following simple result, which states that if~\eqref{eq:mainsys} can be super-linearized, then \i  the vector field $g$ is constant, and \ii  the nonlinear terms of $f$ can be expressed as {\em linear combinations} of the observables.

\begin{theorem}\label{th:main1}
Assume that the system 
\begin{equation}\label{eq:mainsys2}\dot x = f+ug\end{equation}
is super-linearized as  
\begin{equation}\label{eq:canonicalmainsys}
\dot z = A_\ell z + B_\ell u +D_\ell	
\end{equation}
with observables $p:\R^n \to \R^m$. Let $A \in \R^{n \times n},G\in \R^{n \times m},H\in \R^{m \times n}$ and $M\in \R^{m \times m}$ be a block partition of $A_\ell$,  $B \in \R^{n}, C \in \R^m$ a block partition of $B_\ell$ and $D \in \R^{n}, E \in \R^m$ a block partition of $D_\ell$ as

\begin{equation}\label{eq:partabell}
A_\ell = \begin{bmatrix} A & G \\ H & M \end{bmatrix}, B_\ell = \begin{bmatrix} B \\ C
\end{bmatrix}\mbox{ and } 
D_\ell = \begin{bmatrix}
D \\ E
\end{bmatrix}
\end{equation}
Then system~\eqref{eq:mainsys2} is of the form
\begin{equation}\label{eq:canon1}\dot x = A x +G p + B u+D.\end{equation}
\end{theorem}
The proof is elementary, but this result is far reaching in that it will allow us to classify the observables as {\em visible} and {\em hidden}. It also yields the following  Corollary.
\begin{corollary}
	\label{cor:pdepsi}
Assume that the pair $(f,g)$ is super-linearized to $(A_\ell,B_\ell,D_\ell)$ with observables $p$ and $A_\ell,B_\ell, D_\ell$ as in~\eqref{eq:partabell}.  Then $p$ satisfies the pair of partial differential equations
\begin{equation}\label{eq:master1}
\left\lbrace\begin{aligned}&	G\frac{\partial p}{\partial x}(Ax+Gp(x)+D)=G\left(Hx+Mp(x)+E\right) \\
&G\frac{\partial p}{\partial x}B=GC
\end{aligned}
\right.	
\end{equation}
 Furthermore, denoting by  $x(t)$ the solution of~\eqref{eq:mainsys}, by $z(t)$ is the solution of~\eqref{eq:canonicalmainsys} with $z(0)=\iota(x_0)$ and letting $z_2(t)=\bar \Pi z(t)$, we have 
\begin{equation}\label{eq:GpGz}
Gp(x(t))=Gz_2(t).
\end{equation}
\end{corollary}
\noindent  Equation~\eqref{eq:GpGz} highlights the importance of the submatrix $G$ in the partition of $A_\ell$ given in~\eqref{eq:partabell}, we refer to it as the {\em $G$-matrix} of $\cL$.
\begin{definition}[$G$-matrix]
Given a super-linearization $(A_\ell,B_\ell,D_\ell)$ of $(f,g)$ with $A_\ell$ partitioned as in~\ref{eq:partabell}, i.e.,
$$A_\ell=\begin{bmatrix}
 A& G \\ H&M
 \end{bmatrix},
 $$ we call $G \in \R^{n \times m}$ the {\em $G$-matrix} of the super-linearization.	
\end{definition}
 
We can now define precisely what is meant by visible and hidden observables:
 \sloppy
\begin{definition}[Visible and hidden observables]\label{def:vishidobs}
Let $(A_\ell,B_\ell,D_\ell)$ be a super-linearization of~\eqref{eq:mainsys2} via $p:\R^n \to \R^m$ with $A_\ell$ partitioned as in~\eqref{eq:partabell}. Then  $p_j:\R^n \to \R$, $j=1,\ldots,m$, is a {\em visible observable} for the super-linearization if there exists $i \in \{1,\ldots, n\}$   so that $G_{ij} \neq 0$. Otherwise, $p_j$ is a {\em hidden observable}.	We denote by $m_v$ the number of visible observables, and by $m_h$ the number of hidden observables.
\end{definition}
\fussy
The simple example given in the introduction showed that the number of hidden and visible observables is not an intrinsic property of the system, but depends on the choice of super-linearization.

We shall in fact see below, among other results about transformations of super-linearizations,  that the procedure used to go from the first representation in the example above to the second representation, whereby visible observables are concatenated at the expense of increasing the number of hidden observables, can be formalized and used to minimize  the number 
	of visible observables. Two important natural questions regarding hidden and visible observables arise: assuming that a system $\dot x = f+ug$ can be super-linearized, considering all of its super-linearizations, 
	\begin{enumerate}
	\item 	what is the least number of visible observables?
	\item what is the least number of hidden observables?
	\end{enumerate}
We address in this paper the first question, and leave the analysis of the second one to subsequent work. 

An important notion that arises in formulating the answer is the one of  super-linearization in {\em reduced visible form}:

\begin{definition}[Reduced visible form]
We call a super-linearization of $(f,g)$ to $(A_\ell,B_\ell,D_\ell)$ via $p$    in {\em reduced visible form}  if $p$ has no linear nor constant terms, and the visible observables  are {\em linearly independent}. 
\end{definition}
\noindent We recall that the entries of $p:\R^n \to \R^m$ are said to be {\em linearly independent} if  for all $w \in \R^m$ 
$$w^\top p=0 \Leftrightarrow w=0.$$

We now can formulate the second main result of this paper; it asserts the existence of reduced super-linearizations, provided that a super-linearization exists, and relates it to the least number of visible observables of any super-linearization:
\begin{theorem}\label{th:mv}
Assume that $(f,g)$ admits a super-linearization. Then, it admits a super-linearization in {\em reduced visible form}. Furthermore, let $\cL$ be any super-linearization in reduced visible form, and set $m_v^*$  to be the rank of its $G$ matrix. Then, the  least value of the number of visible observables amongst {\em all} super-linearization of $(f,g)$  is $m_v^*$.
\end{theorem}

\section{Proof of the main results}
We now prove Theorems~\ref{th:main1} and \ref{th:mv}  and Corollary~\ref{cor:pdepsi}.

\subsection{Proof of  Theorem~\ref{th:main1} and Corollary~\ref{cor:pdepsi}}

The proof is a simple verification using the definition of super-linearization. 
\begin{proof}[Proof of  Theorem~\ref{th:main1}]
Since the triple $(A_\ell,B_\ell,D_\ell)$ is a super-linearization of $(f,g)$ with observables $p$, it holds that $$e^{t(f+ug)}x = \Pi(e^{t(A_\ell z+B_\ell u+D_\ell)}\iota(x)) \mbox{ for all }x \in \R^n.$$
Differentiating the above relation at $t=0$, we obtain
\begin{equation}\label{eq:pfint1}f(x)+ug(x) = d\Pi (A_\ell z+B_\ell u + D_\ell )|_{z=\iota(x)}.\end{equation}
 Since $\Pi(z_1,\ldots,z_{n+m})=(z_1,\ldots,z_n)$, the Jacobian $d\Pi$ is constant and can be represented in matrix form as $d\Pi=\begin{bmatrix} I_{n \times n} & \mathbf{0}_{n \times m}\end{bmatrix}$. Recalling the partition of $A_\ell$,  $B_\ell$ and $D_\ell$  as 
$$
A=\begin{bmatrix} A & G \\ H & M
\end{bmatrix}, \quad B=\begin{bmatrix} B \\ C\end{bmatrix}\mbox { and } D_\ell = \begin{bmatrix}
D\\E
\end{bmatrix},
$$
we obtain from~\eqref{eq:pfint1} 
$$ 
f(x)+ug(x)= Ax + G p(x) +u B +D \mbox{ for all } x \in \R^n
$$
which proves the statement.
\end{proof}
We now turn to the proof of Corollary~\ref{cor:pdepsi}:

\begin{proof}[Proof of Corollary~\ref{cor:pdepsi}]
Owing to Theorem~\ref{th:main1}, we can write the dynamics as 
\begin{equation}\label{eq:eq1th2}
\dot x = Ax + Gp + B u+D.	
\end{equation}

Let $z(t)=e^{t(A_\ell z+B_\ell u+D_\ell )}\iota(x_0)$ for some arbitrary, but fixed, $x_0$. Recall that $z_1(t):=\Pi z(t) = x(t) $ and $z_2(t):=\bar \Pi(z(t))$. We get
\begin{equation}\label{eq:firsteq}
\dot z_1 = Az_1 + Gz_2+Bu+D = Ax+Gp+B u+D=\dot x.
\end{equation}
Using the fact $z_1(t)=x(t)$ in~\eqref{eq:firsteq}, we obtain
\begin{equation}\label{eq:a12psi}
Gp(x(t))=Gz_2(t).	
\end{equation}
This proves the second part of the statement.

Taking the total time derivative of~\eqref{eq:a12psi}, we get
\begin{equation}\label{eq:pref}
G\frac{\partial p}{\partial x}\left(Ax+Gp(x)+Bu+D\right)=G(Hz_1+Mz_2+Cu+E).
\end{equation}	
Because $x_0$ was arbitrary and because $z(0)=\iota(x_0)$,~\eqref{eq:pref} evaluated at $t=0$ yields the first part of the statement and concludes the proof.
\end{proof}

\subsection{Proof of Theorem~\ref{th:mv}}

Throughout this section, we deal with a system $(f,g)$ as in~\eqref{eq:mainsys} which is assumed to admit a super-linearization. We start this section with a few results which are necessary for the proof of Theorem~\ref{th:mv}, and which may also be of independent interest.

The following lemma is a simple fact about solutions of differential equations and their embeddings in higher-dimensional state-spaces. 

\begin{lemma}\label{lem:diffeq0}
Let $n'>n>0$ be integers and $\psi:\R^n \to \R^{n'}$ an embedding with uniformly bounded above and below derivative. Let $f:\R^n \to \R^n$ and $F:\R^{n'} \to \R^{n'}$ be smooth maps so that $d\psi\cdot f(x)= F(\psi(x))$ for all $x \in \R^n$. Then, for as long as the solutions exist, it holds that
$$\psi(e^{tf}x_0) = e^{tF}\psi(x_0)$$ for all $x_0 \in \R^n$.
\end{lemma}
The result is elementary. To see that it holds, it suffices to note over $\psi(\R^n)$, $F$ and $f$ are related by a change of variables.

\begin{lemma}\label{lem:changevarQ}
	Assume that $(f,g)$ is super-linearizable to 
$(A_\ell, B_\ell, D_\ell)$, with partition given as in~\eqref{eq:partabell} and observables $p:\R^n \to \R^m$. Then, for any $P \in GL(m)$, the system $(f,g)$ can be super-linearized to $(A'_\ell, B'_\ell, D'_\ell)$ with
 \begin{equation}\label{eq:a'b'd'}
 A'_\ell = \begin{bmatrix} A & GP^{-1} \\
 PH & PMP^{-1}	
 \end{bmatrix},
 B'_\ell = \begin{bmatrix}
 B \\
 PC	
 \end{bmatrix}
\mbox{ and }
D'_\ell = \begin{bmatrix}
 D \\ PE	
 \end{bmatrix}
\end{equation}
and observables $p':=Pp.$
\end{lemma}
\begin{proof}
Let $Q \in GL(n+m)$ be given by
\begin{equation}\label{eq:defQ}Q= \begin{bmatrix} I & 0 \\ 0 & P
 \end{bmatrix}.
\end{equation}
 Then, setting $z'=Qz$, a short calculation yields that the super-linearized dynamics in $z'$ variables is given by
\begin{equation}\label{eq:defApP}\dot z' = A'_\ell z' + B_\ell' u + D'_\ell
 \end{equation}
 for $A'_\ell, B'_\ell$ and $D_\ell'$ as in~\eqref{eq:a'b'd'}. 
Let $\iota'(x_0) := \begin{pmatrix} x_0 & Pp(x_0) \end{pmatrix}$. From the definition of $z'$, it holds that
$$ Q e^{t(A_\ell + B_\ell u +D_\ell)} \iota(x_0) = e^{t(A'_\ell + B'_\ell u +D'_\ell)} \iota'(x_0).
$$
From the form of $Q$ in~\eqref{eq:defQ} and the definition of $\Pi$, we see that $\Pi \circ Q = \Pi$;  applying $\Pi$ to both sides of the previous equation, we thus get 
$$x(t) = \Pi e^{t(A'_\ell + B'_\ell u +D'_\ell)} \iota'(x_0),
$$
which concludes the proof.
\end{proof}

The following results show how to create super-linearization with potentially fewer visible observables. 
\sloppy
\begin{proposition}\label{prop:mvrkG}
Assume that the system $(f,g)$  can be super-linearized to $\cL=(A_\ell, B_\ell, D_\ell)$, partitioned as in~\eqref{eq:partabell}, via $p:\R^n \to \R^m$ and with $G$-matrix $G$. Let $m_v$ be the number of its visible observables. Then  $m_v \geq \rank G$. Furthermore, $(f,g)$ admits a super-linearization $\cL'$ (with $G$-matrix $G'$) with  $m_v'=\rank G=\rank G'$ visible observables, which are linear combinations of the entries of $p$.
\end{proposition}
\fussy
The Proposition implies that if a system is super-linearizable, it always admits a super-linearization with $n$ or fewer visible observables.
\begin{proof}
	We first show that $(A_\ell,B_\ell,D_\ell)$ has at least $\rank G$ visible observables. Let $m_v$ be the number of visible observables. Owing to Lemma~\ref{lem:changevarQ} with $P$ a permutation matrix, we can assume without loss of generality that $p_1,\ldots,p_{m_v}$ are visible. Then, by definition, there exists $i_1,\ldots, i_{m_v}$ so that $G_{i_j,j} \neq 0$ for $j=1,\ldots,m_v$ and $G_{ij}=0$ for all $i=1,\ldots,n$ and $j=m_v+1,\ldots,m$. Hence,  the $G$ matrix has {\em exactly}   $m_v$ non-zero columns. It is clear that the number of non-zero columns of $G$ is lower bounded by its rank, which proves the bound.

	We now prove the second part of the statement. Recall the partition of $z\in \R^{n+m}$ as $z=\begin{pmatrix} z_1 & z_2\end{pmatrix}$ described in the notation section. Let $r:=\rank G$. Then, there exists $V \in \R^{n \times r}$ and $W \in \R^{r \times m}$, both of rank $r$, so that  $G=VW$. Introduce the map $\psi:\R^{m+n} \to \R^{m+n+r}$ defined as
	
	$$
	\psi(z)=\begin{bmatrix} z_1 \\ z_2 \\ W z_2\end{bmatrix}
	$$
	and set $y:=\begin{pmatrix} y_1 & y_2 & y_3\end{pmatrix} \in \R^{n+m+r}$, with $y_1 \in \R^n, y_2 \in \R^m$ and $y_3 \in \R^r$. Note that $\psi$ is an embedding.
	
		Now introduce the following linear dynamics in $\R^{m+n+r}$
	\begin{equation}\label{eq:dynzp}
		\frac{d}{dt} \begin{bmatrix} y_1 \\ y_2 \\ y_3 \end{bmatrix} =
		\underbrace{\begin{bmatrix}
		A & 0 & V\\
		H & M & 0 \\
		WH & WM & 0	
		\end{bmatrix}}_{A'_\ell} \begin{bmatrix} y_1 \\ y_2 \\ y_3 \end{bmatrix}+ 
	\underbrace{\begin{bmatrix} B \\ C \\ WC \end{bmatrix}}_{B'_\ell}u+
	\underbrace{\begin{bmatrix} D \\ E \\ WE \end{bmatrix}}_{D_\ell'}=:F(y,u).
	\end{equation}
Denote by $f_\ell(z,u)$ the embedded dynamics $(A_\ell,B_\ell,D_\ell)$. Then we have
	\begin{align*}
d\psi \cdot f_\ell &=	\begin{bmatrix}
	I & 0\\
	0 & I\\
	0 & W	
	\end{bmatrix}\left( \begin{bmatrix} A & VW \\H & M \end{bmatrix}\begin{bmatrix} z_1 \\ z_2\end{bmatrix}+u\begin{bmatrix} B \\C \end{bmatrix}+\begin{bmatrix} D \\ E\end{bmatrix}\right)\\ & = 	\begin{bmatrix}
		A & VW  \\
		H & M  \\
		WH & WM 	
		\end{bmatrix} \begin{bmatrix} z_1 \\ z_2 \end{bmatrix}+ 
	\begin{bmatrix} B \\ C \\ WC \end{bmatrix}u+
	\begin{bmatrix} D \\ E \\ WE \end{bmatrix}\\
	&=\begin{bmatrix}
		A & 0 & V  \\
		H & M &0 \\
		WH & WM &0	
		\end{bmatrix} \begin{bmatrix} z_1 \\ z_2\\ Wz_2 \end{bmatrix}+ 
	\begin{bmatrix} B \\ C \\ WC \end{bmatrix}u+
	\begin{bmatrix} D \\ E \\ WE \end{bmatrix}\\
	&= F(\psi(z),u).
	\end{align*}

Hence, using Lemma~\ref{lem:diffeq0}, the solution of~\eqref{eq:dynzp} initialized at $y(0)=\iota'(x_0):=\begin{bmatrix} x_0 & p(x_0) & Wp(x_0)\end{bmatrix}$ is so that $y_1(t)=z_1(t), y_2(t)=z_2(t)$ and $y_3(t)=Wz_2(t)$. Setting $\Pi'(y):=y_1$, we thus conclude that for all $x_0 \in \R^n$ 
	$$
\Pi'(e^{t(A'_\ell z'+B'_\ell u +D_\ell')}\iota'(x_0)) = e^{t(f+ug)}x_0.
$$
The system can thus be super-linearized to $(A'_\ell,B'_\ell,D'_\ell)$ via  $\begin{pmatrix} p(x)&  Wp(x)\end{pmatrix}$.
\fussy

The $G$-matrix of this super-linearization is $\begin{bmatrix} 0 & V\end{bmatrix}$, and it has rank $r = \rank G$ by definition of $V$. Since $V$ belongs to $\R^{n\times r}$, there are at most $r$ visible observables. Now assume, by contradiction, that there are fewer than $r$ visible observables. Then, $V$ has a column which is identically zero, which contradicts the fact that $\rank V = r$. Furthermore, the visible observables are given by the linear combination $Wp(x)$ of the observables of the orgininal super-linearization. This concludes the proof.
\end{proof}
The above results shows that, at the expense of increasing the number of hidden variables, we can always create a super-linearization with  exactly $\rank G$ visible observables.

We now show how to obtain from a given super-linearization another one with potentially fewer  observables.

\begin{proposition}\label{prop:mh}
Consider the system $(f,g)$ and assume it can be super-linearized to $(A_\ell, B_\ell, D_\ell)$, partitioned as in~\eqref{eq:partabell}. Let $r = \dim O(M,G)$.
Then, $(f,g)$ can be super-linearized with $r$ observables.
\end{proposition}

\begin{proof}
Consider  the auxiliary system
$$	\begin{cases}
\dot w & = Mw+Hu\\ y&= Gw;
\end{cases}
$$
we know from the Kalman observable decomposition~\cite{brockett2015finite, rugh1996linear} that there exist $P \in GL(m)$ so that
$$M':=PMP^{-1} = \begin{bmatrix} M_1 &M_2 \\ 0 & M_3 \end{bmatrix}, \quad G':=GP^{-1} = \begin{bmatrix} 0 & G_1\end{bmatrix}
$$
for some $M_1,M_2, M_3,G_1$ where $M_3 \in \R^{r \times r}$ and $G_1 \in \R^{n \times r}$. Using this $P$ with Lemma~\ref{lem:changevarQ}, and partitioning $z'$ as $\begin{pmatrix} z'_1 & z'_2  & z'_3 \end{pmatrix} $, we get that the system can be super-linearized to
\begin{equation}\label{eq:dynzpred}
	\frac{d}{dt} \begin{bmatrix} z_1' \\ z_2' \\ z_3' \end{bmatrix} =
		\begin{bmatrix}
		A & 0 & G_1\\
		H'_1 & M_1 & M_2 \\
		H'_2 & 0 & M_3	
		\end{bmatrix} \begin{bmatrix} z_1' \\ z_2' \\ z'_3 \end{bmatrix}+ \begin{bmatrix} B \\ C'_1 \\ C'_2  \end{bmatrix}u+
	\begin{bmatrix} D \\ E'_1 \\ E'_2 \end{bmatrix},
\end{equation}
 where $H'_1\in \R^{(m-r) \times n}$, $ C'_1, E'_1 \in \R^{m-r}$  and $H'_2 \in \R^{r \times n}$, $ C'_2, E'_2 \in \R^{r}$ partition $PH, PC$ and $PE$  as $$PH = \begin{bmatrix} H'_1 \\ H'_2 \end{bmatrix}, PC = \begin{bmatrix} C'_1 \\ C'_2 \end{bmatrix}\mbox{ and }PE = \begin{bmatrix} E'_1 \\ E'_2 \end{bmatrix}$$ respectively. It is clear that the dynamics of $z'_1$ and $z'_3$ is independent from $z'_2$. Partitioning $p'(x):=Pp(x)$ as $p'_1(x) \in \R^{m-r}$ and $p'_2 \in \R^{r}$, we thus get from~\eqref{eq:dynzpred}  that the system can be super-linearized to
 
\begin{equation}\label{eq:dynzpred2}
	\frac{d}{dt} \begin{bmatrix} z_1' \\  z_3' \end{bmatrix} =
		\begin{bmatrix}
		A & G_1\\
		H'_2 & M_3	
		\end{bmatrix} \begin{bmatrix} z_1'  \\ z'_3 \end{bmatrix}+ \begin{bmatrix} B \\  C'_2  \end{bmatrix}u+
	\begin{bmatrix} D \\  E'_2 \end{bmatrix},
\end{equation} via the observables $p'_2(x)$. This concludes the proof.
\end{proof}
 
We will now show that if a system admits a super-linearization, it also admits a super-linearization via observables without constant and linear terms. The statement will be a consequence of the following Proposition, which shows that we if we add arbitrary constant and linear terms the observables of a super-linearization, there exists another super-linearization using these modified observables.

\begin{proposition}\label{prop:linearcorrection}
Assume that the system $(f,g)$ can be super-linearized to $(A_\ell,B_\ell, D_\ell)$, partitioned as in~\eqref{eq:partabell}, with observables $p:\R^n \to \R^m$. Let $R \in \R^{m \times n}$, $S \in \R^m$ and let $p'(x):=p(x)+Rx+S$. Then, $(f,g)$ can be super-linearized to $( A'_\ell, B'_\ell,  D'_\ell)$ with observables $p'$, where
\begin{multline}\label{eq:deftildeAell}
A'_\ell = \begin{bmatrix} A-GR & G \\
H-MR+RA-RGR & M+RG
\end{bmatrix},   B'_\ell = \begin{bmatrix} B \\ C+RB
  \end{bmatrix}\\
   \mbox{ and } D '_\ell =\begin{bmatrix}
D-GS\\
E-MS+RD-RGS	+S
\end{bmatrix}
\end{multline}
	
\end{proposition}

\begin{proof}
Let $z \in \R^{n+m}$ be partitioned as $z= \begin{pmatrix} z_1 & z_2 \end{pmatrix}$ with $z_1 \in \R^n$ and $z_2 \in \R^m$. Introduce the variables $y \in \R^{n+m}$ with the same partition $y= \begin{pmatrix} y_1 & y_2 \end{pmatrix}$ and set 
$$
\begin{cases}
	y_1 &= z_1 \\
	y_2 &= z_2 + Rz_1+S.
\end{cases}
$$
A simple calculation show that 
$$
\dot y =  A'_\ell y +  B'_\ell u + D'_\ell.
$$
Now set $\iota'(x) = \begin{pmatrix} x &   p' (x) \end{pmatrix}^\top.$ Because $y(t)$ and $z(t)$ are related by a change of variables that sends $z_1$ to $y_1$, it follows that 
$$ \Pi (e^{t(A_\ell z + B_\ell u+D_\ell )}\iota(x_0)) = \Pi(e^{t( A'_\ell y +  B'_\ell u + D'_\ell) } \iota'(x_0)),
$$
which proves the statement.
\end{proof}

Using Proposition~\ref{prop:linearcorrection} with $S=p(0)$ and $R=\frac{d}{dx}|_0p(x)$ proves the following Corollary.
\begin{corollary}\label{cor:slnolinear}
	If the system $(f,g)$  can be super-linearized, then it can be super-linearized via observables without constant or linear terms.
\end{corollary}

The next result shows how to remove linearly dependent visible observables  from a super-linearization.

\begin{proposition}
	\label{prop:reducingLIm}
	Assume that $(f,g)$ is super-linearized to $(A_\ell,B_\ell,D_\ell)$ via $p:\R^n \to \R^m$, with $G$-matrix of rank $m_v$. Assume there are $m_v$ visible observables, of which $m_v'\leq m_v$   are linearly independent. Then, $(f,g)$ can be super-linearized via  observables $p':\R^n \to \R^{m'}$, with $m'=m_h+m_v'$. 
\end{proposition}
\begin{proof}
\sloppy	The proof is again constructive. Using Lemma~\ref{lem:changevarQ} with an appropriately chosen permutation matrix,  we can assume without loss of generality that the last $m_v$ observables $p_{m-m_v+1},\ldots,p_{m}$ are  visible,  the first $m_v'$ of which  $p_{m-m_v+1},\ldots,p_{m-m_v+1+m_v'}$ are linearly independent. To simplify the notation below, we correspondingly partition $p$ as $p=\begin{pmatrix} p_h & p_v' & p_v\end{pmatrix}$, where $p_h$ are the hidden observables and $p_v'$ a maximal set of linearly independent visible observables. Because $m_v'\leq m_v$ visible observables are independent, there exists  $Q \in \R^{(m_v-m_v')\times m_v'}$ so that 

	\begin{equation}\label{eq:defpp'}
	\begin{bmatrix}p_h \\ p_v' \\ p_v \end{bmatrix}=\underbrace{\begin{bmatrix} I & 0 \\ 0 & I \\0& Q \end{bmatrix}}_{V} \begin{bmatrix} p_h \\ p_v' \end{bmatrix}.
	\end{equation}
	\fussy
Let $W:\R^m \to \R^{m'}$ be the canonical projection onto the first $m'$ coordinates,   then \begin{equation}\label{eq:p'wp}p':=\begin{bmatrix}p_h\\p_v'\end{bmatrix}=W\begin{bmatrix}p_h\\p_v'\\p_v\end{bmatrix},\end{equation}
 and $\overline W:\R^m \to \R^{m-m'}$ be the canonical projection onto the last $m-m'$ coordinates (so that $p_v = \bar W p$).

	 We partition $z_2 \in \R^m$ similarly to $p$ as $z_2 =\begin{pmatrix} z_h & z_v' & z_v \end{pmatrix}$ and set $z_2'=\begin{pmatrix} z_h & z_v'  \end{pmatrix}$.
	 From the assumption on the rank of $G$ on the partition of $p$ described above, there exists $\bar G \in \R^{n \times m_v}$  of full column rank so that $G$ is of the form 
	\begin{equation}\label{eq:formGG1}G=\begin{bmatrix} 0 & \bar G \end{bmatrix}.
	\end{equation}

With these preliminaries, introduce $y:=\begin{pmatrix} y_1 & y_2\end{pmatrix}$ with $y_1\in \R^n$ and $y_2 \in \R^{m'}$ and define the dynamics
	\begin{equation}\label{eq:dynVmp}
					\frac{d}{dt} \begin{bmatrix} y_1 \\ y_2 \end{bmatrix} =
		\underbrace{\begin{bmatrix}
		A & GV \\
		WH & WMV 
		\end{bmatrix}}_{A'_\ell} \begin{bmatrix} y_1 \\ y_2 \end{bmatrix}+ 
	\underbrace{\begin{bmatrix} B \\ WC  \end{bmatrix}}_{B'_\ell}u+
	\underbrace{\begin{bmatrix} D \\ WE \end{bmatrix}}_{D_\ell'}=:f_\ell(y,u),
	\end{equation}
	where $V$ was defined in~\eqref{eq:defpp'}.
We also let $\dot z = F_\ell(z,u)$ be the affine dynamics induced by $(A_\ell,B_\ell,D_\ell,p)$, as described in~\eqref{eq:canonicalmainsys}, which we repeat here for convenience:

\begin{equation}\label{eq:repeatdynz}
\frac{d}{dt} \begin{bmatrix} z_1 \\ z_2
 \end{bmatrix}	=
\begin{bmatrix} A & G \\ H & M \end{bmatrix}\begin{bmatrix} z_1 \\ z_2
 \end{bmatrix}+\begin{bmatrix} B \\ C
\end{bmatrix} u + \begin{bmatrix}
D \\ E
\end{bmatrix}=:F_\ell(z,u)
\end{equation}

In order to apply Lemma~\ref{lem:diffeq0}, we introduce the map
$$\psi: \R^{n+m'} \to \R^{n + m}: \begin{bmatrix} y_1 \\ y_2 \end{bmatrix} \mapsto \begin{bmatrix}
	y_1\\Vy_2
\end{bmatrix}.$$ We claim that that $d\psi\cdot  f(y,u) = F(\psi(y),u)$. To see that the claim holds, note that on the one hand,

\begin{equation}\label{eq:dpsif}
	d\psi \cdot f_\ell(y,u) = \begin{bmatrix}
		A & GV \\
		VWH & VWMV 
		\end{bmatrix} \begin{bmatrix} y_1 \\ y_2 \end{bmatrix}+ 
	\begin{bmatrix} B \\ VWC  \end{bmatrix}u+
	\begin{bmatrix} D \\ VWE \end{bmatrix} 
\end{equation}
and on the other hand 
\begin{equation}\label{eq:Fpsi}
F(\psi(y),u)	=
\begin{bmatrix} A & G \\ H & M \end{bmatrix}\begin{bmatrix} y_1 \\ Vy_2
 \end{bmatrix}+\begin{bmatrix} B \\ C
\end{bmatrix} u + \begin{bmatrix}
D \\ E
\end{bmatrix}.
\end{equation}
The first (block) rows of~\eqref{eq:dpsif} and~\eqref{eq:Fpsi} are clearly equal. It thus suffices to show that the corresponding second (block) rows are equal as well. 

Now recall that by~\eqref{eq:GpGz} in Corollary~\ref{cor:pdepsi}, it holds that $Gz_2=Gp$. Using the form of $G$ given in~\eqref{eq:formGG1},  the fact that $\bar G$ is full column rank  and~\eqref{eq:defpp'}  we have that
$$ \begin{bmatrix}z_v' \\ z_v \end{bmatrix}=\begin{bmatrix} p_v' \\ p_v \end{bmatrix} = \begin{bmatrix} I \\ Q \end{bmatrix} p_v'.
$$
 From the previous equation,  we conclude that
\begin{equation}\label{eq:zvzvp}
z_v(t) = Q z_v'(t) \mbox{ for all } t \geq 0	
\end{equation}
and thus
\begin{equation}\label{eq:zvzvp2}
z_2(t) = V z_2'(t) = V  W  z_2(t) \mbox{ for all } t \geq 0.	
\end{equation}
In particular, $\dot z_2 = VW \dot z_2$. This shows that the second block rows of~\eqref{eq:dpsif} and~\eqref{eq:Fpsi} are equal and proves the claim. From Lemma~\ref{lem:diffeq0}, we thus conclude that the solutions of~\eqref{eq:dynVmp} and~\eqref{eq:repeatdynz} are conjugate. From the definition of $\psi$, it thus holds that $y_1(t)=z_1(t)$ for all $t \geq 0$. We conclude that $(A_\ell',B_\ell',D_\ell', p')$ is a super-linearization of $(f,g)$.  
 \end{proof}

\begin{proposition}\label{prop:invarkG}
	Let $\cL$ and $\cL'$ be super-linearizations in reduced visible form of the same system and denote by  $G$ and $G'$ their respective  $G$-matrices. Then $$\rank G= \rank G'.$$
\end{proposition}

\begin{proof}
	Let $\cL=(A_\ell,B_\ell,D_\ell,p)$ and $\cL'=(A_\ell',B_\ell',D_\ell',p')$ be any two super-linearizations in reduced visible form of the same system $(f,g)$, and $G$, $G'$ their respective $G$-matrices. We claim that

\begin{equation}\label{eq:eqGG'}
 Gp(x) = G'p'(x) \mbox{ for all } x \in \R^n.
 \end{equation}

 To see that the equality holds, we apply Theorem~\ref{th:main1} to $\cL$ and $\cL'$ to obtain  that
\begin{equation}\label{eq:equalofcanons}
	 A x +G p + B u+D = A' x +G' p' + B' u+D' 
\end{equation}
for all $x \in \R^n$. Now, because  because neither $p$ nor $p'$ has constant terms, evaluating~\eqref{eq:equalofcanons} at $x=0$ yields $D=D'$. Furthermore, differentiating~\eqref{eq:equalofcanons} and evaluating the result at $x=0$, we get
$$
 A  +G \frac{\partial p}{\partial x}|_{0}  = A'  +G' \frac{\partial p'}{\partial x}|_{0}.  
$$ Since neither $p$ nor $p'$ have linear terms, we conclude that $A=A'$. Finally, equating the terms containing the control $u$, we get that $B=B'$, which proves~\eqref{eq:eqGG'}.

To proceed, denote by $m_v$ and $m_v'$ the number of visible observables in $\cL$ and $\cL'$ respectively. Recalling the fact that hidden observables correspond to zero columns of a $G$-matrix, we denote by  $G_v \in \R^{n \times m_v}$ (resp. $G_v' \in \R^{n \times m_v'}$)   the submatrix of $G$ (resp. $G'$) obtained by removing its zero columns and by  $p_v:\R^n \to \R^{m_v}$ (resp. $p_v':\R^n \to \R^{m_v'}$) the subvector of $p_v$ (resp. $p'_v$) obtained by removing the hidden observables. With this notation,~\eqref{eq:eqGG'} yields
\begin{equation}\label{eq:eqGG'2}
	G_vp_v(x) = G_v'p_v'(x) \mbox{ for all } x \in \R^n.
\end{equation}

We now assume by contradiction that $\rank G_v=\rank G < \rank G'=\rank G'_v$. Then, since $G'_v$ has more linearly independent columns than $G_v$, there exists a nonzero vector $w \in \R^n$ so that $w^\top G_v=0$ but $w^\top G'_v =: v' \neq 0$. Multiplying~\eqref{eq:eqGG'2} on the left by $w^\top$, we get
$$0 = v' p_v'(x) \mbox{ for all } x \in \R^n,
$$
which contradicts the linear independence of the visible observables of $\cL'$. The same reasoning shows that $\rank G > \rank G'$ also yields a contradiction, which leads to $\rank G= \rank G'$ and concludes the proof.
\end{proof}

We are now in a position to prove  Theorem~\ref{th:mv}.

\begin{proof}[Proof of Theorem~\ref{th:mv}]
We prove the two claims of the Theorem using the following steps. Let $\cL$  be an arbitrary super-linearization of $(f,g)$.  Then, we 

\begin{enumerate}
\item[(1)]	remove the linear and constant terms from the observables and show it can be done without affecting the rank of the corresponding $G$-matrix. We call the so-obtained super-linearization $\cL_1$. 
\item[(2)] We obtain from $ \cL_1$  another super-linearization with linearly independent visible observables. We show that these observables  are  still without linear nor constant terms and that the rank of $G$ can only {\em decrease} in the process.  We call $\cL_2$ the so-obtained super-linearization. 
\end{enumerate}
The super-linearization $\cL_2$ obtained using items (1) and (2) is then in reduced visible form, which proves the first part of Theorem~\ref{th:mv}.
In fact, these two items provide a procedure which assigns to an arbitrary super-linearization a   super-linearization in reduced visible form, and the rank of the corresponding $G$-matrix can only decrease. Proceeding,
\begin{enumerate}

\item[(3)]  Using the fact that  any two distinct  super-linearizations in {\em reduced visible form} of the same system have $G$-matrices of the same rank, we conclude  that no super-linearization can have fewer than $\rank G$ visible observables for $G$ associated to a super-linearization in reduced visible form.	
\end{enumerate}

Let $\cL =(A_\ell,B_\ell,D_\ell,p)$ be a super-linearization of~\eqref{eq:mainsys}. 

\xc{Proof of (1):} Denote by $\cL_1$ the super-linearization obtained using Proposition~\ref{prop:linearcorrection} on $\cL$ and by $G_1$ its $G$-matrix. From~\eqref{eq:deftildeAell}, we see that $G_1$ is equal to the $G$-matrix  associated with $\cL$, which proves the first item.
\xc{Proof of (2):} Denote by $\cL_1'$ the super-linearization obtained using Proposition~\ref{prop:mvrkG} on $\cL_1$ and by $G_1'$ its $G$-matrix. Then $\rank(G_1')=\rank(G_1)$ by construction, and  the number of visible observables of $\cL_1'$ is {\em smaller than} the number of visible observables of $\cL_1$. Since these visible observables are linear combinations of the original observables, they are also without linear and constant terms. 

Next, apply Proposition~\ref{prop:reducingLIm} on $\cL_1'$ to obtain $\cL_2$, whose $G$-matrix is denoted by $G_2$. Then $\cL_2$ has linearly independent visible observables and since the observables in $\cL_2$ are a subset of the observables in $\cL_1'$ per~\eqref{eq:p'wp}, the observables in $\cL_2$ have no linear nor constant terms. 
From~\eqref{eq:dynVmp} in the proof of Proposition~\ref{prop:reducingLIm}, we have  $G_2=G_1' V$, where $V$ is the matrix defined in~\eqref{eq:defpp'}.   It then follows that $\rank G_2 \leq \min(\rank(G_1'),\rank (V)) \leq \rank G_1'$, which concludes the proof of item 2.

The super-linearization $\cL_2$ has linearly independent visible observables without linear nor constant terms and is thus in {\em reduced visible} form, proving the first statement of Theorem~\ref{th:mv}.

\xc{Proof of (3):} Let $\cL'$ be the super-linearization in reduced-visible form  obtained from $\cL$ using items (1) and (2) and $r'=\rank G'$, where $G'$ is the $G$-matrix associated with $\cL'$.  
We can create, using Proposition~\ref{prop:mvrkG}, a super-linearization $\cL^*$ with exactly $r'$ visible observables and a $G$-matrix of rank $r'$. 
Now assume that there exists another lifted-realization of $(f,g)$, say $\bar \cL$, with associated $G$-matrix of rank $\bar r < r'$. Then, using the reduction technique of items (1) and (2), we obtain from $\bar \cL$ super-linearization $\bar \cL '$ in reduced visible form whose corresponding $G$-matrix has rank {\em at most} $\bar r$. But this contradicts Proposition~\ref{prop:invarkG}, which states that the ranks of the $G$-matrices of {\em all} reduced super-linearization are equal. This shows that the minimal number of visible observables amongst all super-linearizations of~\eqref{eq:mainsys} is $r'$, the rank of the $G$ matrix of any reduced super-linearization. This concludes the proof.
\end{proof}

\section{Summary and outlook}

Super-linearizations provide exact linear representations---or embeddings as a linear systems---of nonlinear dynamical systems at the expense of an increase in the dimension of the state-space. This increase is required to accommodate the evolution of the observables necessary for the linearization of the flow. We have in this paper introduced a classification of the observables in terms of whether they appear explicitly in the dynamics, the so-called visible observables, or whether they support the linearization without being explicitly present in the original system, the so-called hidden observables. We then provided several methods which, given that a super-linearization exist, allow to create  new super-linearizations of a system. Distinct super-linearizations of the same system can be defined on state-spaces of different dimensions, and an important question that arises is to quantify the least dimension of a super-linearization. This amounts to finding the embedding with the least number of observables. We have provided in this paper an answer to this question for the case of visible observables, by providing a lower-bound for the number of observables amongst all super-linearizations of a system as well as a procedure to obtain a super-linearization realizing this bound. In subsequent work, we will address the case of hidden observables and, by extension, the question of the least dimension of a super-linearization of a given system.

\bibliographystyle{plain}
\bibliography{carleref.bib}

\end{document}